\documentclass[12pt,reqno]{amsart}
\usepackage[headings]{fullpage}
\usepackage{amssymb,bbm}
\usepackage{graphicx}
\usepackage{texdraw}
\usepackage{url}
\usepackage[all]{xy}
\usepackage{float}   %%% for floating figures
\usepackage[bookmarks=true,%
    colorlinks=true,%
    linkcolor=blue,%
    citecolor=blue,%
    filecolor=blue,%
    menucolor=blue,%
    urlcolor=blue,%
    breaklinks=true]{hyperref}

\usepackage{multirow}
\usepackage{mathtools}
\usepackage{xcolor}
\usepackage{comment}
\usepackage{tikz,tikz-cd}
\usetikzlibrary{knots, hobby, decorations.pathreplacing,
shapes.geometric, calc}

\newtheorem{theorem}{Theorem}

\theoremstyle{definition}

\newtheorem{lemma}[theorem]{Lemma}

\def\BN{\mathbbm N}
\def\BZ{\mathbbm Z}
\def\BQ{\mathbbm Q}
\def\BR{\mathbbm R}

\def\BC{\mathbbm C}

\def\calE{\mathcal E}

\def\s{\sigma}

\def\SL{\mathrm{SL}}

\def\tq{\tilde{q}}
\def\Res{\mathrm{Res}}
\def\Vol{\mathrm{Vol}}

\def\a{\alpha}

\def\={\;=\;}
\def\+{\,+\,}
 
\def\-{\,-\,}

\def\be{\begin{equation}}
\def\ee{\end{equation}}

\def\Li{\mathrm{Li}}

\def\Im{\mathrm{Im}}

\def\eu{e}
\def\ramuno{i}

\newcommand{\dif}{\mathop{}\!\mathrm{d}}

\makeatletter
\newcommand{\spx}[1]{%
  \if\relax\detokenize{#1}\relax
    \expandafter\@gobble
  \else
    \expandafter\@firstofone
  \fi
  {^{#1}}%
}
\makeatother

\newcommand{\od}[3][]{\frac{\dif\spx{#1}#2}{\dif#3\spx{#1}}}

\tikzset{
  knot diagram/every strand/.append style={
    ultra thick,
    red
  },
  show curve controls/.style={
    postaction=decorate,
    decoration={show path construction,
      curveto code={
        \draw [blue, dashed]
        (\tikzinputsegmentfirst) -- (\tikzinputsegmentsupporta)
        node [at end, draw, solid, red, inner sep=2pt]{};
        \draw [blue, dashed]
        (\tikzinputsegmentsupportb) -- (\tikzinputsegmentlast)
        node [at start, draw, solid, red, inner sep=2pt]{}
        node [at end, fill, blue, ellipse, inner sep=2pt]{}
        ;
      }
    }
  },
  show curve endpoints/.style={
    postaction=decorate,
    decoration={show path construction,
      curveto code={
        \node [fill, blue, ellipse, inner sep=2pt] at (\tikzinputsegmentlast) {}
        ;
      }
    }
  }
}

\makeatletter

\renewcommand\thepart{\@Roman\c@part}%
\renewcommand\part{%
   \if@noskipsec \leavevmode \fi
   \par
   \addvspace{6.7ex}%
   \@afterindentfalse
   \secdef\@part\@spart}
\def\@part[#1]#2{%
    \ifnum \c@secnumdepth >\m@ne
      \refstepcounter{part}%
      \addcontentsline{toc}{part}{Part~\thepart.\ #1}%
    \else
      \addcontentsline{toc}{part}{#1}%
    \fi
    {\parindent \z@ \raggedright
     \interlinepenalty \@M
     \normalfont
     \ifnum \c@secnumdepth >\m@ne
       \centering\large\scshape \partname~\thepart.%
       \hspace{1ex}%
     \fi%
     \large\scshape #2%
     \markboth{}{}\par}%
    \nobreak
    \vskip 4.7ex
    \@afterheading}
  \def\@spart#1{
  \refstepcounter{part}%
  \addcontentsline{toc}{part}{#1}%
    {\parindent \z@ \raggedright
     \interlinepenalty \@M
     \normalfont
     \centering\large\scshape #1\par}%
     \nobreak
     \vskip 4.7ex
     \@afterheading}
\renewcommand*\l@part[2]{%
  \ifnum \c@tocdepth >-2\relax
    \addpenalty\@secpenalty
    \addvspace{0.75em \@plus\p@}%
    \begingroup
      \parindent \z@ \rightskip \@pnumwidth
      \parfillskip -\@pnumwidth
      {\leavevmode
       \normalsize \bfseries #1\hfil \hb@xt@\@pnumwidth{\hss #2}}\par
       \nobreak
       \if@compatibility
         \global\@nobreaktrue
         \everypar{\global\@nobreakfalse\everypar{}}%
      \fi
    \endgroup
  \fi}

\def\l@subsection{\@tocline{2}{0pt}{2pc}{6pc}{}}
\makeatother

\begin{document}
\title[Algebraic aspects of holomorphic quantum modular forms]{
  Algebraic aspects of holomorphic quantum modular forms}

\author{Ni An}
\address{
  Department of Mathematics \\
  Southern University of Science and Technology \\
  Shenzhen, China}
\email{2315687238@qq.com}
\author{Stavros Garoufalidis}
\address{% Max Planck Institute for Mathematics \\
         % Vivatsgasse 7, 53111 Bonn, GERMANY \newline
  International Center for Mathematics, Department of Mathematics \\
  Southern University of Science and Technology \\
  Shenzhen, China}
\email{stavros@mpim-bonn.mpg.de}
\author{Shana Yunsheng Li}
\address{
  International Center for Mathematics, Department of Mathematics \\
  Southern University of Science and Technology \\
  Shenzhen, China \newline
  {\tt \url{https://li-yunsheng.github.io}}}
\email{yl202@illinois.edu}

\thanks{
  {\em Keywords and phrases}: knots, links, 3-manifolds, Jones polynomial,
  hyperbolic 3-manifolds, Kauffman bracket, skein theory, Witten's Conjecture,
  Habiro ring, Quantum Modularity Conjecture, Volume Conjecture, Faddeev quantum
  dilogarithm, state integrals, holomorphic quantum modular forms.   
}

\date{11 July 2024}%5 March 2024}%\today}
%\dedicatory{}

\begin{abstract}
  Matrix-valued holomorphic quantum modular forms are intricate objects
  associated to 3-manifolds (in particular to knot complements) that arise
in successive refinements of the Volume Conjecture of knots and involve three
holomorphic, asymptotic and arithmetic realizations. It is expected that the algebraic
properties of these objects can be deduced from the algebraic properties of descendant
state integrals, and we illustrate this for the case of the $(-2,3,7)$-pretzel knot.
\end{abstract}

\maketitle

{\footnotesize
\tableofcontents
}

%%%%%%%%%%%%%%%%%%%%%%%%%%%%%%%%%%%%%%%%%%%%%%%%%%%%%%%%%%%%%%%%%%%%%%%%%%%%
%%%%%%%%%%%%%%%%%%%%%%%%%%%%%%%%%%%%%%%%%%%%%%%%%%%%%%%%%%%%%%%%%%%%%%%%%%%%

\section{Introduction}
\label{sec.intro}

The Volume Conjecture of Kashaev links the asymptotics of the Jones polynomial
of a hyperbolic knot and its parallels with the hyperbolic geometry of the
knot complement~\cite{K97}. Explicitly, the conjecture (combined with the results
of Murakami--Murakami~\cite{MM}) asserts that for a
hyperbolic knot $K$ in 3-space~\cite{Th}, we have
\be
\lim_{N \to \infty} \frac{1}{N} \log | J^K_N(e^{2 \pi i/N})|
= \frac{\Vol(S^3\setminus K)}{2 \pi}
\ee
where $J^K_N(q) \in \BZ[q^{\pm 1}]$ is the Jones polynomial of $K$, colored with
the $N$-dimensional irreducible representation of $\mathfrak{sl}_2(\BC)$, and
normalized to be $1$ for the unknot. The definition of the colored Jones polynomial,
that we omit, may be found in~\cite{RT:ribbon}

The Volume Conjecture is considered one of the main problems of quantum topology.
Although it is currently known only for a handful of knots, it can be
strengthened in numerous ways to include a statement about asymptotics to all orders
in $N$, and with exponentially small terms included. One of these successive
refinements 
of the quantum modularity conjecture for the Kashaev invariant of a knot lead to the
concept of a matrix-valued holomorphic quantum modular forms introduced and studied
in~\cite{GZ:kashaev,GZ:qseries}. The latter are rather intricate objects that involve
matrices of 
\begin{itemize}
\item
  holomorphic objects, that is $q$-series with integer coefficients convergent when
  $|q| \neq 1$.
\item
  asymptotic/analytic objects, that is factorially divergent formal
  power series that are Borel resummable and whose Stokes phenomenon is explained
  in terms of the $q$-series above.
\item
  arithmetic objects, that is collections of functions defined near each
  complex root of unity that arithmetically determine each other $p$-adically.
\end{itemize}

This sounds like a daunting collection of objects (associated for example to knots)
that come from different worlds
and are somehow stitched together. Despite this, it turns out that matrix-valued
holomorphic quantum modular forms have algebraic aspects that can be formulated
and proven using the algebraic properties of a variant of the Andersen--Kashaev
invariants~\cite{AK}, namely the descendant state integrals.

The three simplest hyperbolic knots are the $4_1$ knot, the $5_2$ knot and the
$(-2,3,7)$ pretzel knot~\cite{Th,snappy} shown in Figure~\ref{f.3knots}.

\begin{figure}[htpb!]
\centering
\scalebox{0.6}{
\begin{tikzpicture}[use Hobby shortcut]
\begin{knot}[
      consider self intersections=true,
    %  draft mode=crossings,
      ignore endpoint intersections=false,
      flip crossing=4,
      only when rendering/.style={
      %  show curve endpoints
      }
    ]
    \strand[black] ([closed]0,0) .. (-.8,0.8) .. (.8,2.4) .. (-.8,4)
    .. (-2.4,2.4) .. (0,.8) .. (2.4,2.4) .. (0.8,4) .. (-.8,2.4)
    .. (.8,0.8) .. (0,0);
\end{knot}
%\path (0,-.7);
\end{tikzpicture}
}
\qquad
\scalebox{0.6}{
\begin{tikzpicture}[use Hobby shortcut]
\begin{knot}[
      consider self intersections=true,
    %  draft mode=crossings,
      ignore endpoint intersections=false,
      flip crossing/.list={6,4,2},
      only when rendering/.style={
    %    show curve endpoints
        }
      ]
    \strand[black] ([closed]1.6,1.6) .. (1.4,0) .. (-1.8,-.8)
    .. (.4,.8) .. (-1.6,1.6) .. (-1.4,0) .. (1.8,-.8) .. (-.4,.8) .. (1.6,1.6);
\end{knot}
\end{tikzpicture}
}
\qquad
\scalebox{0.6}{
\begin{tikzpicture}[use Hobby shortcut]
\begin{knot}[
            consider self intersections=true,
        %    draft mode=crossings,
            ignore endpoint intersections=false,
            flip crossing/.list={3,5,7,9,11,14},
            only when rendering/.style={
            %  show curve endpoints
              }
            ]
     \strand[black] ([closed]0,.2) .. (2,-.2) .. (4,0) .. (4.2,.2)
     .. (3.6,.6) .. (4.2,1.2) .. (3.6,1.8) .. (4.2,2.4) .. (3.6,3) .. (4.2,3.6)
     .. (3.6,4) .. (3,3.8) .. (3,3) .. (2.4,2.2) .. (3,1.4) .. (2.4,.8)
     .. (.4,.8) .. (-.2,2) .. (.4,3) .. (2.4,3) .. (3,2.2) .. (2.4,1.4)
     .. (3,.8) .. (3,.4) .. (3.6,.2) .. (4.2,.6) .. (3.6,1.2) .. (4.2,1.8)
     .. (3.6,2.4) .. (4.2,3) .. (3.6,3.6).. (4.2,3.8) .. (4,4.2) .. (2,4.4)
     .. (0,4) .. (-.2,3) .. (.4,2) .. (-.2,.8) ..(0,.2);
\end{knot}
\end{tikzpicture}
}
\caption{The three simplest hyperbolic knots from left to right: $4_1$, $5_2$
  and the $(-2,3,7)$-pretzel knot.}
\label{f.3knots}
\end{figure}
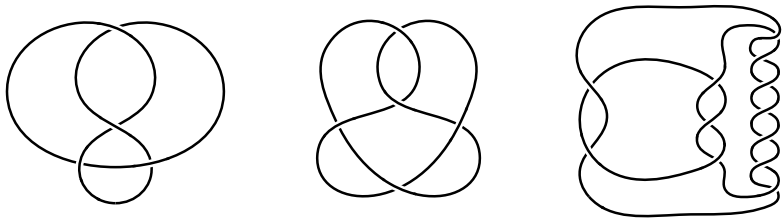

The corresponding matrices of the $4_1$ and the $5_2$ knot are $2 \times 2$ and
$3 \times 3$ due to the fact that all boundary-parabolic $\SL_2(\BC)$-representations
are conjugates of the geometric representation which is defined over $\BQ(\sqrt{-3})$
and over the cubic field of discriminant $-23$. The
holomorphic quantum modular forms of the first two were studied in detail in
~\cite{GZ:kashaev,GZ:qseries} as well as
in~\cite{GGM:resurgent,GGM:peacock,GGMW:trivial}. 

On the other hand, it was pointed out in~\cite[Sec.2.1]{GZ:kashaev} that the
corresponding matrices associated to the $(-2,3,7)$ pretzel knot are $6 \times 6$ due
to the fact that the geometric representation is defined over the same number field
as for $5_2$, but in addition to that there are three additional 
boundary-parabolic $\SL_2(\BC)$-representations defined over $\BQ(2\cos(2\pi/7))$.
In this paper we will study the algebraic aspects of the holomorphic quantum
modular forms associated to the $(-2,3,7)$-pretzel knot completing the partial
work of~\cite{GZ:kashaev,GZ:qseries}. Explicitly, we will show that

\begin{itemize}
\item[(a)]
  The factorization of the descendant state integral defines a $6 \times 6$
  matrix of (deformed) $q$-hypergeometric series; 
%  $W_\lambda(q)$ of $q$-series for $|q| \neq 1$ which is given by
%  $q$-hypergeometric series and their deformations;
  see Theorem~\ref{thm.1}. 
\item[(b)]
  The matrix is a fundamental solution of a self-dual linear
  $q$-difference equation; see Theorem~\ref{thm.2}.
\item[(c)]
  The corresponding cocycle is a holomorphic function that extends from
  $\tau \in \BC\setminus \BR$ to the cut-plane $\BC'=\BC\setminus (-\infty,0]$;
  see Theorem~\ref{thm.4}.
\item[(d)]
  The stationary phase of the descendant state integral determines a
  $6 \times 6$ matrix of asymptotic series; see Theorem~\ref{thm.5}.  
\end{itemize}

Along with the detailed definitions and proofs, we will give an explanation of
why the proofs work from first principles. 

%%%%%%%%%%%%%%%%%%%%%%%%%%%%%%%%%%%%%%%%%%%%%%%%%%%%%%%%%%%%%%%%%%%%%%%%%%%%
%%%%%%%%%%%%%%%%%%%%%%%%%%%%%%%%%%%%%%%%%%%%%%%%%%%%%%%%%%%%%%%%%%%%%%%%%%%%

\section{The descendant state integral}
\label{sec.state}

\subsection{The state integral}

A key role in our paper is the state integral invariant of 3-dimensional
manifolds with torus boundary (in particular knot complements) introduced by
Andersen--Kashaev~\cite{AK}. This is a multi-dimensional 
integral whose integrand is a product of the Faddeev quantum
dilogarithm $\Phi_b(x)$~\cite{Faddeev} times an exponential of a quadratic form,
assembled out of a suitable triangulation of the manifold.

% As is well-known, and we will review in Section~\ref{sub.proof.fac}, 

In the case of the $(-2,3,7)$-pretzel knot, the Andersen--Kashaev state integral
is a four-dimensional state integral that can be reduced to the following
one-dimensional state integral as shown in~\cite[Eqn.(58)]{GK:evaluation}:

\begin{equation}
\label{Zdef}  
Z_{(-2,3,7)}(\tau) \= 
\int_{\mathbb{R}+\ramuno \frac{c_b}{2}+\ramuno\varepsilon}
\Phi_{\sqrt{\tau}}(x)^2\Phi_{\sqrt{\tau}}(2x-c_b)
\eu^{-\pi \ramuno(2x-c_b)^2}\dif x 
\end{equation}
where $\tau=b^2$ and $c_b=\ramuno(b+b^{-1})/2$.
Two key properties of this absolutely convergent state integral are that
\begin{itemize}
\item[(a)]
  it defines a holomorphic function on the cut-plane
  $\BC' = \BC\setminus (-\infty,0]$, and
\item[(b)]
  when $\tau \in \BC\setminus\BR$, it can be factorized as a sum of products of
  $q$-series and $\tq$-series where $q = \eu^{2\pi \ramuno \tau}$
  and $\tilde{q} = \eu^{-2\pi \ramuno/\tau}$.
\end{itemize}
This factorization property is expected to hold for all state integrals that
appear naturally in complex Chern--Simons theory of knots and 3-manifolds as explained
in~\cite{holomorphic-blocks,GK:qseries} and in the case of the above state integral,
it was given in~\cite[Prop.12]{GZ:qseries}.

The main reason behind this factorization is algebraic and follows from the
quasi-periodicity of the Faddeev quantum dilogarithm which implies that it is
a meromorphic function with poles in the lattice points of a two-dimensional cone
and prescribed residues. Upon applying the residue theorem, the sum over the cone
(which is coupled by the exponential of a quadratic form with integer coefficients)
decouples due to the fact that $e^{2 \pi i k}=1$ for all integers $k$.
% For a more detailed discussion, see Section~\ref{sub.proof.fac} below. 
  
\subsection{The descendant state integral and its $q$-series}

A descendant version of the state integral~\eqref{Zdef}
obtained by inserting in the integrand of~\eqref{Zdef} the exponential of a linear
form in two integer variables $\lambda,\lambda' \in \mathbb{Z}$, in an analogous
way as was done for the $4_1$ and $5_2$ knots in~\cite[Sec.4.3]{GGM:resurgent}.
The descendant state integral of $(-2,3,7)$ pretzel knot is

\begin{equation}
\label{Zdesc}  
Z_{(-2,3,7)}^{(\lambda,\lambda' )}(\tau)
\= \int_{\mathbb{R}+\ramuno \frac{c_b}{2}+\ramuno\varepsilon}
\Phi_{\sqrt{\tau}}(x)^2\Phi_{\sqrt{\tau}}(2x-c_b)
\eu^{-\pi \ramuno(2x-c_b)^2 +2\pi (\lambda b-\lambda'  b^{-1})x}\dif x \,.
\end{equation}

This integral can be factorized as a finite sum of products of $q$-series and
$\tq$-series for the same reason that the integral~\eqref{Zdef} does. 
Deforming the contour of integration upwards, applying the residue theorem,
collecting the residues and observing the same decoupling as was done
in~\cite{GK:qseries,GK:evaluation}, we obtain the factorization into $q$-series.

\begin{theorem}
\label{thm.1}
For all $\tau \in \BC\setminus\BR$, we have
\be
\label{eq:q-series}
\begin{aligned}
2\eu^{\frac{\pi \ramuno}{4}} &
q^{-\frac{\lambda}{2}}\tilde{q}^{-\frac{\lambda' }{2}}
Z^{(\lambda,\lambda' )}_{(-2,3,7)}(\tau) \\ 
& \hspace{-1.5cm} = -\frac{1}{2\tau} h_{\lambda,0}(\tau)h_{\lambda',2}(\tau^{-1})
+h_{\lambda,1}(\tau)h_{\lambda',1}(\tau^{-1})
- \frac{\tau}{2}h_{\lambda,2}(\tau)h_{\lambda',0}(\tau^{-1})
\\ 
&\hspace{-1.5cm} -\ramuno \Big(\frac{1}{2}h_{\lambda,3}(\tau)h_{\lambda',4}(\tau^{-1})
-\frac{1}{2}h_{\lambda,4}(\tau)h_{\lambda',3}(\tau^{-1})
+h_{\lambda,5}(\tau)h_{\lambda',5}(\tau^{-1}) \Big) \,.
\end{aligned}
\ee
\end{theorem}

In the above theorem
%$h_{\lambda,j}(\tau) =H_{\lambda,j}(e^{2 \pi i \tau})$ for
%$\tau \in \BC\setminus\BR$ where $H_{\lambda,j}(q)$ are $q$-series defined
%in Section~\ref{sub.defqseries} for $|q| \neq 1$ by
\be
\label{hH}
h_{\lambda,j}(\tau)\coloneqq H_{\lambda,j}(\eu ^{2\pi \ramuno \tau}),
\quad
H_{\lambda,j}(q) \=
\begin{cases}
  H^+_{\lambda,j}(q) & \text{if} \,\, |q|<1 \\
  (-1)^{\delta_j}H^-_{-\lambda,j}(q^{-1}) & \text{if} \,\, |q|>1
\end{cases}
\ee
where $H_{\lambda,j}(q)$ are $q$-series defined in Section~\ref{sub.proof.fac}
for $|q| \neq 1$ and $\delta=(0,1,2,0,0,0)$ is a weight vector. $H^{\pm}_{\lambda,j}(q)$
are power series of $q^{1/8}$ whose first few terms are given by
\begin{tiny}
\be
\label{Hqvalues}  
\begin{aligned}
H_{0,0}^+(q)&\=1+q^3+3q^4+7q^5+13q^6+\cdots  &   
H_{0,0}^-(q)&\=1+q^2+3q^3+7q^4+13q^5+\cdots
\\
H_{0,1}^+(q)&\=1-4q-8q^2-3q^3+3q^4+\cdots  &   
H_{0,1}^-(q)&\=1-4q-5q^2+q^3+7q^4+\cdots 
\\
H_{0,2}^+(q)&\=\frac{2}{3}-6q+6q^2+\frac{242}{3}q^3+200q^4+\cdots &
H_{0,2}^-(q)&\=\frac{5}{6}-10q+\frac{17}{6}q^2+\frac{141}{2}q^3+\frac{971}{6}q^4+\cdots
\\
H_{0,3}^+(q)&\=q^{1/8}(q+2q^{3/2}+4q^2+6q^{5/2}+\cdots)&
H_{0,3}^-(q)&\=q^{-1/8}(q+2q^{3/2}+4q^2+6q^{5/2}+\cdots)
\\
H_{0,4}^+(q)&\=1+q^3-q^4+3q^5-3q^6+\cdots&
H_{0,4}^-(q)&\=1+q^2-q^3+3q^4-3q^5+\cdots
\\
H_{0,5}^+(q)&\=q^{1/8}(q-2q^{3/2}+4q^2-6q^{5/2}+\cdots)&
H_{0,5}^-(q)&\=q^{-1/8}(q-2q^{3/2}+4q^2-6q^{5/2}+\cdots) 
\end{aligned}  
\ee
\end{tiny}

\subsection{A self-dual linear $q$-difference equation}
\label{sub.selfdual}

As we saw in the previous section, the factorization of the state
integral~\eqref{Zdesc} produced
six sequences $H_{\lambda,j}(q)$ of $q$-hypergeometric series for $j=0,\dots,5$
indexed by $\lambda \in \BZ$ and defined for $|q| \neq 1$. We now show that these
six sequences are solutions of a common sixth order linear $q$-difference equation.
The algebraic reason for this is that the integrand of the descendant state integral
is a $q$-holonomic function of three variables $x$, $\lambda$ and $\lambda'$, as
follows from the quasi-periodicity of the Faddeev quantum dilogarithm
(see Equation~\eqref{2quasi} below). Zeilberger
theory implies that the integral is a $q$-holonomic function of $\lambda$ and
$\lambda'$~\cite{WZ}. Due to the factorization of the integral, it follows that its
$\lambda$-dependent part is a $q$-holonomic function. An alternative explanation for
the linear $q$-difference equation would be to use the explicit $q$-hypergeometric
formulas for the six $q$-series. In fact, in this case we can do the corresponding
algebraic calculation by elementary telescoping methods and obtain the following.

\begin{theorem}
\label{thm.2}
For each $j=0,\ldots,5$, the sequence $H_{\lambda,j}(q)$ for $|q| \neq 1$
and $\lambda\in\BZ$ satisfies the linear $q$-difference equation
\begin{equation}
\label{Hqdiff}
\begin{aligned}
y_{\lambda +6}(q) + 2 \, y_{\lambda+5}(q) -(q+q^{\lambda+4}) \, y_{\lambda+4}(q)
-2(q+1) \, y_{\lambda+3}(q)
\\
& \hspace{-5.5cm}
- y_{\lambda+2}(q)+ 2q \, y_{\lambda+1} (q) +q \, y_\lambda(q)=0 \,.
\end{aligned}
\end{equation}
\end{theorem}

%% see Mathematica file: 237Pretzel.State.Integral.qSeries.Computation.nb

Consider the truncated Wronskian 
\begin{equation}
\label{eq:W-def}
W_\lambda(q)=\left(H_{\lambda+i,j}(q) \right)_{0\leq i,j\leq 5} \qquad
|q| \neq 1 \,
\end{equation}
of the six solutions to the $q$-difference equation~\eqref{Hqdiff}. Technically,
the Wronskian is a $\BZ \times 6$ matrix whose block indexed by the raws
corresponding to $i=0,\dots,5$ is the above matrix $W_\lambda$.
We next give an orthogonality property of the truncated Wronskian, which implies that
the six sequences of $q$-series form a fundamental solution set of~\eqref{Hqdiff}
and satisfy quadratic relations. 

\begin{theorem}
\label{thm.3}
The determinant of the truncated Wronskian is given by
\begin{equation}
\label{detW}  
\det(W_\lambda(q))  = 32 q^{\lambda+\frac{11}{4}} \,.
\end{equation}
The truncated Wronskian satisfies the orthogonality property
\begin{tiny}
\begin{equation}
\label{eq:orthogonality}
W_\lambda (q) \left(
\begin{matrix}
0 & 0 & \frac{1}{2} & 0 & 0 &0\\
0 & 1 & 0 & 0 & 0 & 0\\
\frac{1}{2} & 0 &0 &0 &0 &0\\
0 & 0 &0 &-1 &0 &0\\
0 &0 &0 &0 &\frac{1}{4} & 0\\
0 & 0 &0 &0 &0 &-1
\end{matrix}
\right)W_{-\lambda -5}(q^{-1})^T= \left(\begin{matrix}
-12 & 8 & -4 & 2 & 0 & 0\\
8 & -4 & 2 & 0 & 0 &0\\
-4 & 2 & 0 & 0& 0 &2\\
2 & 0 & 0 & 0 & 2 & -4\\
0 & 0 & 0 &2 & -4 & 8+2q^{\lambda +2}\\
0 & 0 & 2 & -4 & 8+2q^{\lambda +3} & -12 - 4q^{\lambda +2}-4q^{\lambda +3}
\end{matrix}\right).
\end{equation}
\end{tiny}
\end{theorem}

Equations~\eqref{detW} and~\eqref{eq:orthogonality} were first guessed by explicit
computations of the $q$-series. But once they were guessed, they were proven
algebraically, i.e., by reducing them to identities among rational functions in
$\BQ(q,q^\lambda)$. This concludes our last algebraic aspect of matrix-valued
holomorphic quantum modular forms, discussed in detail in Section~\ref{sub.thmW}
below.
  
A consequence of Equation~\eqref{eq:orthogonality} (in fact, of its $(1,6)$-entry)
is that the collection of $q$-series $H_{\lambda,j}^{\pm}(q)$ satisfies the
quadratic relation 

\begin{equation}
\label{quadrel}
\begin{aligned}
\frac{1}{2}H_{\lambda ,0}^{+}(q)H_{\lambda ,2}^{-}(q)
-H_{\lambda ,1}^{+}(q)H_{\lambda ,1}^{-}(q)+
\frac{1}{2}H_{\lambda ,2}^{+}(q)H_{\lambda ,0}^{-}(q)
\\ & \hspace{-8cm}
-H_{\lambda ,3}^{+}(q)H_{\lambda ,3}^{-}(q)
+\frac{1}{4}H_{\lambda ,4}^{+}(q)H_{\lambda ,4}^{-}(q)
-H_{\lambda ,5}^{+}(q)H_{\lambda ,5}^{-}(q)=0 \,.
\end{aligned}
\end{equation}

A second consequence of Equation~\eqref{eq:orthogonality} (and in fact, an
equivalent statement to it) is that the $q$-holonomic module associated with
the linear $q$-difference equation~\eqref{Hqdiff} is self-dual. For a detailed
definition of self-dual $q$-holonomic modules, we refer the reader
to ~\cite[Sec 2.5]{GW:qmod}.

\subsection{A cocyle}

The last topic of our paper concerns the analytic continuation of
a cocycle, defined in two forms below (Equations ~\eqref{eq:descendant-matrix} and
~\eqref{eq:cocycle}), from a holomorphic function on $\BC\setminus\BR$ to one on
$\BC'=\BC\setminus(-\infty,0]$. This remarkable statement, which concerns the
holomorphic aspects of matrix-valued holomorphic quantum modular forms, follows
immediately from the factorization of the descendant state integrals
(Theorem~\ref{thm.1}), the self-duality property (Theorem~\ref{thm.3}) and the fact
that state integrals are holomorphic functions in the cut-plane $\BC'$. 

\begin{theorem}
\label{thm.4}
\rm{(a)} The matrix-valued function 
\begin{equation}
\label{eq:descendant-matrix}
F_{\lambda,\lambda' }(\tau)\coloneqq W_{-\lambda' -5}(\tilde{q}^{-1})\left( \begin{matrix}
0 & 0 &-\frac{\tau}{2} &0 &0 & 0\\
0 & -1 &0 &0 &0 & 0\\
-\frac{1}{2\tau} & 0 &0 &0 &0 &   0\\
0 & 0 &0 &0 &\frac{\ramuno}{2} & 0\\0 & 0 &0 &-\frac{\ramuno}{2} &0 & 0\\
0 & 0 &0 &0 &0 & -\ramuno 
\end{matrix} \right)W_\lambda(q)^{T}
\end{equation}
defined for $\tau = b^2 \in \mathbb{C} \backslash  \mathbb{R} $, has entries given
by the descendant state integrals up to a prefactor given by~\eqref{eq:q-series},
and therefore extends to a holomorphic function on the cut plane $\BC'$.
\newline
\rm{(b)}
The matrix-valued function 
\begin{equation}
\label{eq:cocycle}  
W_{\lambda,\lambda' }(\tau) \coloneqq \left(W_{\lambda'} (\tilde{q})^T\right)^{-1}
\left( \begin{matrix}
-\frac{1}{\tau} & 0 &0 &0 &0 & 0\\
0 & -1 &0 &0 &0 & 0\\
0 & 0 &-\tau &0 &0 &   0\\
0 & 0 &0 &0 &-\frac{\ramuno}{2} & 0\\0 & 0 &0 &-2\ramuno &0 & 0\\
0 & 0 &0 &0 &0 & \ramuno 
\end{matrix} \right) W_\lambda(q)^{T}
\end{equation}
extends to a holomorphic function of $\tau\in \mathbb{C}'$. 
\end{theorem}

Note that Equation~\eqref{eq:cocycle} follows
from ~\eqref{eq:descendant-matrix} and~\eqref{eq:orthogonality}.

To make contact with the results of the paper~\cite{GW:qmod},
Equation~\eqref{eq:cocycle} becomes the cocycle $U(-1/\tau)^{-1} D(\tau) U(\tau)$
where $U_\lambda(\tau)=(W_\lambda(e^{2 \pi i \tau})^T)^{-1}$ and $D(\tau)$ is the
automorphy factor in the middle matrix of the right hand side of~\eqref{eq:cocycle}.

%%%%%%%%%%%%%%%%%%%%%%%%%%%%%%%%%%%%%%%%%%%%%%%%%%%%%%%%%%%%%%%%%%%%%%%%%%%%
%%%%%%%%%%%%%%%%%%%%%%%%%%%%%%%%%%%%%%%%%%%%%%%%%%%%%%%%%%%%%%%%%%%%%%%%%%%%

\section{The $q$-series}
\label{sec.qseries}

\subsection{Algebraic properties of the Faddeev quantum dilogarithm}
\label{sub.faddeev}

As stated in the introduction, all of our theorems follow from algebraic
properties of state integrals, which in turn follow from some well-known properties
of the Faddeev quantum dilogarithm function; see ~\cite{Faddeev} as well
as~\cite[App.A]{AK}. In this section we review these properties briefly and
highlight their algebraic aspects.

To begin with, the Faddeev quantum dilogarithm function is defined by
\be
\label{faddef}
\Phi_b(x) \= \exp\Big( \frac{1}{4} \int_{\BR+i \epsilon}
\frac{e^{-2ixw}}{\sinh(bw)\sinh(b^{-1}w)}
\frac{dw}{w} \Big)
\ee
When $\Im(b^2)>0$, the Faddeev quantum dilogarithm is given by
a ratio of two infinite Pochhammer symbols as follows

\begin{equation}
\label{fad}
\Phi_b(x)
=\frac{(e^{2 \pi b (x+c_b)};q)_\infty}{
(e^{2 \pi b^{-1} (x-c_b)};\tq)_\infty} \,,
\end{equation}
where
\be
\label{qqtcb}
q=e^{2 \pi i b^2}, \qquad 
\tq=e^{-2 \pi i b^{-2}}, \qquad
c_b=\frac{i}{2}(b+b^{-1}), \qquad \Im(b^2) >0 \,. 
\ee
Remarkably, the ratio~\eqref{fad} admits an extension to all values of $b$ with 
$b^2 \in \BC\setminus (-\infty,0]$. $\Phi_b(x)$ is a meromorphic function of $x$
with poles in the set $c_b + i \BN b + i \BN b^{-1}$. In other words, 
the poles are given
by $x_{m,n} = i\big(m+\frac{1}{2})b + i\big(n+\frac{1}{2}) b^{-1}$ with $m$ and $n$
nonnegative integers.

The Faddeev quantum dilogarithm satisfies the quasi-periodicity
\be
\label{2quasi}
\frac{\Phi_b(x+c_b+ib)}{\Phi_b(x+c_b)} \= \frac{1}{1-q e^{2 \pi b x}}, 
\qquad
\frac{\Phi_b(x+c_b+ib^{-1})}{\Phi_b(x+c_b)} \= 
\frac{1}{1-\tq^{-1} e^{2 \pi b^{-1} x}} \,.
\ee
Quasi-periodicity among other things, explains the structure of the set of poles, and
that the residue of $\Phi_b(x)$ at the pole $x_{m,n}$ is given
by~\cite[Lem.2.1]{GK:qseries} 
\be
\label{resxmn}
\Res_{x=x_{m,n}} \Phi_b(x) = -\frac{b}{2\pi} \frac{(q;q)_\infty}{(\tq;\tq)_\infty}
\frac{1}{(q;q)_m}
\frac{1}{(\tq^{-1};\tq^{-1})_n} \,.
\ee
Notice that the poles are parametrized by a pair $(m,n)$ of natural numbers and the
residue decouples, i.e., it is the product of a function of $m$ times a function of $n$.

Although it will not play in our paper, we mention that the Faddeev quantum
dilogarithm satisfies the functional equation
\be
\label{Phi2term}
\Phi_b(x) \Phi_b(-x)=e^{\pi i x^2} \Phi_b(0)^2, 
\qquad
\Phi_b(0)^2=q^{\frac{1}{24}} \tq^{-\frac{1}{24}}
\ee
which implies that its set of zeros is the negative of its set of poles, 
and also allows us to move $\Phi_b(x)$ from the denominator to the numerator
of the integrand of a state-integral. 

\subsection{Factorization}
\label{sub.proof.fac}
  
To prove Theorem~\ref{thm.1}, we observe that the integrand is a meromorphic
function of $x$. We then deform the contour of integration upwards, apply the residue
theorem and collect residues. The poles of the Faddeev quantum dilogarithms were
discussed in the previous section, and they are parametrized by the lattice points
$x_{m,n}$ in a two dimensional cone, and the residues are decoupled functions of
$m$ (involving $q$) and $n$ (involving $\tq$), see Equation~\eqref{resxmn}. This
decoupling persists when we evaluate the exponential function
$\eu^{-\pi \ramuno(2x-c_b)^2 +2\pi (\lambda b-\lambda'  b^{-1})x}$ at $x_{m,n}$. Thus,
the sum over the lattice points $(m,n)$
of a two-dimensional lattice becomes a product over the lattice points $m$ of
one-dimensional lattice times a product over the lattice points $n$ of
the other one-dimensional lattice. Taking into account that the integrand
of~\eqref{Zdesc} is the product of three Faddeev quantum dilogarithms, this gives the
proof of Theorem~\ref{thm.1}. 

The details of the factorization of the original state integral~\eqref{Zdef}
were given in~\cite[section~A.6]{GZ:qseries}. Following the same proof mutatis
mutandis and inserting the integers $\lambda$ and $\lambda'$, produces the definition
of the $q$-series given in Equations~\eqref{H012def} and~\eqref{H345def} below and
concludes the proof
of Theorem~\ref{thm.1}
\qed

The corresponding series are given as follows. We define $H_{\lambda ,j}^\pm(q)$
for $|q|<1$ and $j=0,1,2$ by:
\begin{equation}
\label{H012def}
H_{\lambda ,j}^+(q) \=(-1)^{\lambda}\sum_{m=0}^\infty t_{\lambda,m}(q)p_{\lambda,j,m}(q),
\quad
H_{\lambda' ,j}^-(q) \=
(-1)^{\lambda' } \sum_{n=0}^\infty T_{\lambda' ,n}P_{\lambda',j ,n}(q),
\end{equation}
with
\begin{equation}
\label{t012def}  
t_{\lambda,m}(q)= \frac{q^{m(2m+1)+\lambda m} }{(q;q)_m^2 (q;q)_{2m}},
\quad T_{\lambda' ,n}(q) =\frac{q^{n(n+1)+\lambda'  n}}{ (q;q)_n^2(q;q)_{2n}},
\end{equation}
and 
\begin{equation}
\label{p012def}
\begin{aligned}
&p_{\lambda,0,m}(q)= 1,\quad p_{\lambda,1,m}(q)
= 4m+\lambda +1 -2E_1^{(m)}(q) -2 E_1^{(2m)} (q) ,\quad \\
&p_{\lambda,2,m}(q) =p_{\lambda,1,m}(q)^2 -2E_2^{(m)}(q) - 4E_2^{(2m)}(q)
-\frac{1}{3}\mathcal{E}_2(q), \\
&P_{\lambda',0 ,n}(q)=1,\quad P_{\lambda',1 ,n}(q)
=2n+\lambda'  +1 -2E_1^{(n)}(q) -2E_1^{(2n)}(q),\\ 
&P_{\lambda',2 ,n}(q)= P_{\lambda',1 ,n}(q)^2 +12 E_2^{(0)}(q)
- \frac{1}{2} - 2E_2^{(n)}(q) - 4E_2^{(2n)}(q) +\frac{1}{3}\mathcal{E}_2(q),
\end{aligned}
\end{equation}

\noindent
and for $j=3,4,5$ by:

\begin{tiny}
\begin{equation}
\label{H345def}
\begin{aligned}
H_{\lambda,3}^+(q) & =  \frac{(-1)^\lambda q^{1/8}}{(1-q^{1/2})^2}
\sum_{m=0}^\infty \frac{q^{(2m+1)(m+1)+\lambda(m+1/2)}}{(q^{3/2};q)_m^2(q;q)_{2m+1}}
& H_{\lambda' ,4}^-(q) & =
\sum_{n=0}^\infty \frac{q^{n(n+1)+\lambda'  n}}{(-q;q)_n^2 (q;q)_{2n}}
\\
H_{\lambda,4}^+(q) & =  \sum_{m=0}^\infty \frac{q^{(2m+1)m+\lambda m}}{(-q;q)_m^2(q;q)_{2m}}
& H_{\lambda' ,3}^-(q) & = \frac{(-1)^{\lambda' }q^{-1/8}}{(1-q^{-1/2})^2}\sum_{n=0}^\infty
\frac{q^{n(n+2)+\lambda' (n+1/2)}}{(q^{3/2};q)_n^2(q;q)_{2n+1}}
\\
H_{\lambda,5}^+(q) & = \frac{q^{1/8}}{(1+q^{1/2})^2}\sum_{m=0}^\infty
\frac{q^{(2m+1)(m+1)+\lambda(m+1/2)}}{(-q^{3/2};q)_m^2(q;q)_{2m+1}}  &H_{\lambda' ,5}^- (q)
& = \frac{q^{-1/8}}{(1+q^{-1/2})^2}\sum_{n=0}^\infty
\frac{q^{n(n+2)+\lambda' (n+1/2)}}{(-q^{3/2};q)_n^2(q;q)_{2n+1}} \,.
\end{aligned}
\end{equation}
\end{tiny}

Here, $\mathcal{E}_1(q)$ and $\mathcal{E}_2(q)$ denote the Eisenstein series
of weights 1 and 2
\begin{equation}
\label{E1E2}
\mathcal{E}_1(q)  = 1-4\sum_{n=1}^\infty \frac{q^n}{(1-q^n)}, \qquad   
\mathcal{E}_2(q)  = 1-24\sum_{n=1}^\infty \frac{q^n}{(1-q^n)^2}
\end{equation}
and
\be
\label{Elm}
E_{l}^{(m)}(q)=\sum_{s=1}^\infty \frac{s^{l-1}q^{s(m+1)}}{1-q^s} 
\ee
are some series that appear in the factorization of one-dimensional
state integrals~\cite{GK:qseries}.

When $(\lambda,\lambda' )=(0,0)$, this factorization can be connected to that
in \cite[eq.~(52)]{GK:evaluation} (see Appendix~\ref{sub.comparison} below). 

% For computation of these q-series in Mathematica, see:
% 237-Pretzel/237Pretzel.State.Integral.qSeries.Computation.nb

\subsection{The linear $q$-difference equation}

In this section we prove Theorem~\ref{thm.2}. We will use elementary telescoping
summation than the more elaborate methods of~\cite{WZ}.

We begin with the case $j=0$ and $|q|<1$, hence
$$
H_{\lambda,0}(q) = H_{\lambda,0}^+(q)
= (-1)^\lambda\sum_{m=0}^\infty \frac{q^{m(2m+1)+\lambda m} }{(q;q)_m^2 (q;q)_{2m}} \,.
$$
Since $(q;q)_m = \prod_{i=1}^m (1-q^i)$, we have
\begin{equation*}
\begin{aligned}
q^\lambda \sum_{m=0}^\infty \frac{q^{m(2m+1)+\lambda m} }{(q;q)_m^2 (q;q)_{2m}}
=& \sum_{m=0}^\infty  \frac{q^{m(2m+1)+\lambda (m+1)} }{(q;q)_m^2 (q;q)_{2m}} =
\sum_{m=1}^\infty  \frac{q^{(m-1)(2m-1)+\lambda m} }{(q;q)_{m-1}^2 (q;q)_{2m-2}} \\
=& \sum_{m=1}^\infty  \frac{q^{m(2m+1)+\lambda m} }{(q;q)_{m}^2 (q;q)_{2m}}
\frac{(1-q^m)^2(1-q^{2m-1})(1-q^{2m})}{q^{4m-1}} \,.
\end{aligned}
\end{equation*}
Since $1-q^m =0$ when $m=0$, we can replace the summation in the above equation
from $m=0$ to $m=\infty$. 
%\begin{equation*}
%\sum_{m=1}^\infty  \frac{q^{m(2m+1)+\lambda m} }{(q;q)_{m}^2 (q;q)_{2m}}
%\frac{(1-q^m)^2(1-q^{2m-1})(1-q^{2m})}{q^{4m-1}} = \sum_{m=0}^\infty
%\frac{q^{m(2m+1)+\lambda m} }{(q;q)_{m}^2 (q;q)_{2m}}
%\frac{(1-q^m)^2(1-q^{2m-1})(1-q^{2m})}{q^{4m-1}}.
%\end{equation*}
Since
\begin{equation}
\label{simp-q-id}
\frac{(1-q^m)^2(1-q^{2m-1})(1-q^{2m})}{q^{4m-1}}
=q^{1-4m} - 2q^{1-3m} -q^{-2m} +2q^{1-m} +2q^{-m} - q - 2q^m + q^{2m},
\end{equation}
we obtain
\begin{small}
\begin{equation*}
\begin{aligned}
q^{\lambda} H_{\lambda,0}^+(q) =& (-1)^\lambda  \sum_{m=0}^\infty
\frac{q^{m(2m+1)+\lambda m} }{(q;q)_{m}^2 (q;q)_{2m}}
(q^{1-4m} - 2q^{1-3m} -q^{-2m} +2q^{1-m} +2q^{-m} - q - 2q^m + q^{2m}) \\
= & qH_{\lambda-4,0}^+(q) +2q H_{\lambda-3,0}^+(q) -H_{\lambda-2,0}^+(q)
-(2+2q)H_{\lambda-1,0}^+(q) - q H_{\lambda,0}^+(q) +2  H_{\lambda+1,0}^+(q)
+ H_{\lambda+2,0}^{+}(q) \,.
\end{aligned}
\end{equation*}
\end{small}
This gives the $q$-difference equation for $H_{\lambda,j}(q)$ when $j=0$ and $|q|<1$.
Similarly one proves the $q$-difference equation for the cases $j=0,3,4,5$ and
whenever $|q|\neq 1$. 

For $j=1$ and $|q|<1$, we have 
$$
H_{\lambda,1}(q)=H_{\lambda,1}^+(q) = (-1)^\lambda \sum_{m=0}^\infty
\frac{q^{m(2m+1)+\lambda m} }{(q;q)_m^2 (q;q)_{2m}} p_{\lambda,m}^{(1)}(q) \,,
$$
where
\begin{equation*}
p_{\lambda,1,m}(q)  = 4m+\lambda+1 -2E_1^{(m)}(q)-2E_1^{(2m)}(q). 
\end{equation*}
Hence
\begin{equation*}
\begin{aligned}
qH_{\lambda-4,1}^+(q)+2qH_{\lambda-3,1}^+(q)-H_{\lambda-2,1}^+(q)
-(2+2q)H_{\lambda-1,1}^+(q)&-qH_{\lambda,1}^+(q)+2H_{\lambda+1,1}^+(q)+H_{\lambda+2,1}^+(q) 
\\
&= (-1)^\lambda \sum_{m=0}^\infty
\frac{q^{m(2m+1)+\lambda m} }{(q;q)_m^2 (q;q)_{2m}}g_{\lambda,m}(q),
\end{aligned} 
\end{equation*}
where 
\begin{equation}
\label{gtemp}
\begin{aligned}
g_{\lambda,m}(q) = & q^{1-4m}p_{\lambda-4,1,m}(q)
- 2q^{1-3m} p_{\lambda-3,1,m}(q)-q^{-2m}p_{\lambda-2,1,m}(q) 
\\
& + (2+2q)q^{-m}p_{\lambda-1,1,m}(q)
-qp_{\lambda,1,m}(q)-2q^m p_{\lambda+1,1,m}(q)+q^{2m}p_{\lambda+2,1,m}(q).
\end{aligned}
\end{equation}
We are going to show that $(-1)^\lambda \sum_{m=0}^\infty \frac{q^{m(2m+1)+\lambda m} }{
(q;q)_m^2 (q;q)_{2m}}g_{\lambda,m}(q) = q^\lambda H_{\lambda,1}^+(q)$. 
Noticing the recursive relation that
\begin{equation}
\label{recE1m}
E_{1}^{(m)}(q)-E_{1}^{(m-1)}(q)=\sum_{s=1}^{\infty}\left(
\frac{q^{s(m+1)}}{1-q^s}-\frac{q^{sm}}{1-q^s} \right)=\sum_{s=1}^\infty -q^{sm}
=-\frac{q^m}{1-q^m},
\end{equation}
we convert $p_{\lambda,1,m}(q)$ into the following form
\begin{equation}
\label{recp1}
\begin{aligned}
p_{\lambda,1,m}(q) =& 4m+\lambda +1 -2E_{1}^{(m-1)}(q) - 2E_1^{(2m-2)}(q)
+ \frac{2q^m}{1-q^m}+\frac{2q^{2m-1}}{1-q^{2m-1}}+\frac{2q^{2m}}{1-q^{2m}}\\
=& 4(m-1)+\lambda +1 -2E_{1}^{(m-1)}(q) - 2E_1^{(2m-2)}(q)+ f_{1,m}(q) + 4 \\
=& p_{\lambda,1,m-1}(q) + f_{1,m}(q) + 4,
\end{aligned}
\end{equation}
where 
\begin{equation}
\label{eq:def-f-m-1}
f_{1,m}(q) \coloneqq
\frac{2q^m}{1-q^m}+\frac{2q^{2m-1}}{1-q^{2m-1}}+\frac{2q^{2m}}{1-q^{2m}}. 
\end{equation}
Substituting the ~\eqref{recp1} into ~\eqref{gtemp}, combining the common factors
$p_{\lambda,1,m-1}(q) +f_{1,m}(q)$ and applying the identity ~\eqref{simp-q-id},
we see that
\begin{equation*}
\begin{aligned}
  g_{\lambda,m}(q) =&  \frac{(1-q^m)^2(1-q^{2m-1})(1-q^{2m})}{q^{4m-1}}
  \left(p_{\lambda,1,m-1}(q)+f_{1,m}(q)\right)\\
&- \left(2q^{1-3m} + 2q^{-2m} -6(q+1)q^{-m} +4q +10q^m -6q^{2m}\right). 
\end{aligned}
\end{equation*}
Since
\begin{equation*}
  \frac{(1-q^m)^2(1-q^{2m-1})(1-q^{2m})}{q^{4m-1}} f_{1,m}(q)
  = 2q^{1-3m} + 2q^{-2m} -6(q+1)q^{-m} +4q +10q^m -6q^{2m},
\end{equation*}
we conclude that 
\begin{equation*}
  g_{\lambda,m}(q) =  \frac{(1-q^m)^2(1-q^{2m-1})(1-q^{2m})}{q^{4m-1}}
  p_{\lambda,1,m-1}(q).
\end{equation*}
Therefore
\begin{equation*}
\begin{aligned}
  (-1)^\lambda \sum_{m=0}^\infty \frac{q^{m(2m+1)+\lambda m} }{(q;q)_m^2 (q;q)_{2m}}
  g_{\lambda,m}(q) =& (-1)^\lambda \sum_{m=1}^\infty
  \frac{q^{(m-1)(2m-1)+\lambda m} }{(q;q)_{m-1}^2 (q;q)_{2m-2}}p_{\lambda,1,m-1}(q)\\
  =& (-1)^\lambda q^\lambda \sum_{m=0}^\infty
  \frac{q^{m(2m+1)+\lambda m} }{(q;q)_m^2 (q;q)_{2m}} p_{\lambda,1,m}(q)
  = q^\lambda H_{\lambda,1}^+(q),
\end{aligned}
\end{equation*}
as desired. Similarly one proves the $q$-difference equation for $j=1,2$ and
$|q|\neq 1$, using the recursive relation ~\eqref{recE1m} and
\begin{equation}
  \label{recE2m}
  E_2^{(m)}(q)- E_2^{(m-1)}(q)
  = \sum_{s=1}^\infty \left(\frac{sq^{s(m+1)}}{1-q^s} - \frac{sq^{sm}}{{1-q^s}} \right)
  = \sum_{s=1}^\infty - sq^{sm} = -\frac{q^m}{(1-q^m)^2}.
\end{equation}
This completes the proof of Theorem~\ref{thm.2}.
\qed

\subsection{Self-duality}
\label{sub.thmW}

In this section we prove Theorem~\ref{thm.3}. Throughout this
section we assume $|q|<1$ and give the proof for this case only; the proof
for $|q|>1$ is similar and is omitted. Our method can be used to give a systematic
proof of the self-duality properties of the $q$-holonomic modules that appear in
the refined quantum modularity conjecture of knot complements or of closed 3-manifolds.

We first compute the determinant of the truncated Wronskian $W_\lambda(q)$.
It is well-known
(see eg~~\cite[Lemma~4.7]{IrreducibilityQDiff}) that it satisfies the first order
linear $q$-difference equation
\begin{equation*}
\det(W_{\lambda+1}(q))  - q \det( W_{\lambda}(q))  = 0 \,.
\end{equation*}
It follows that $\det(W_{\lambda}(q)) = q^{\lambda} c(q)$ for some $q$-series $c(q)$ 
independent of $\lambda$. We claim that
\begin{equation}
\label{detW-Oq}
\det(W_\lambda(q))  = 32q^{\lambda+11/4} +O(q^{3\lambda/2}),
\end{equation}
for all sufficiently large natural numbers $\lambda$, which implies
that $c(q) = 32q^{11/4}$. To show~\eqref{detW-Oq}, 
recall that $W_{\lambda}(q) = \left(H_{\lambda+i,j}^+(q)\right)_{0\leq i,j\leq 5}$
when $|q|<1$. The definition of $H_{\lambda,j}^+(q)$ implies that
\begin{equation}
\label{HR}  
  H^\pm_{\lambda,j}(q) = R^\pm_{\lambda,j}(q) + O(q^{3\lambda/2}),
%  \quad H^-_{\lambda+5-i,j}(q) = R^-_{\lambda+5-i,j}(q) +O(q^{3\lambda/2}),
\end{equation}
where $R^+_{\lambda,j}(q)$ and $R^-_{\lambda,j}(q)$ are given by 
\be
\label{R+}
\begin{aligned}
R^+_{\lambda,j}(q) &=(-1)^{\lambda} \left(p_{\lambda,j,0}(q)
+ p_{\lambda,j,1}\frac{q^{\lambda+3}}{(1-q)^4(1+q)}\right),\quad j=0,1,2,
\\  
R^+_{\lambda,3}(q) &=(-1)^{\lambda}\frac{q^{1/8}}{(1-q^{1/2})^2}
\frac{q^{1+\lambda/2}}{1-q},
\\  
R^+_{\lambda,4}(q) &=1+ \frac{q^{\lambda+3}}{(1+q)^3(1-q)^2},
\\  
R^+_{\lambda,5}(q) &=\frac{q^{1/8}}{(1+q^{1/2})^2} \frac{q^{1+\lambda/2}}{1-q},
\\  
\end{aligned}
\ee
and
\be
\label{R-}
\begin{aligned}
R^-_{\lambda,j}(q) &=(-1)^{\lambda} \left(P_{\lambda,j,0}(q)
+ P_{\lambda,j,1}(q)\frac{q^{\lambda+2}}{(1-q)^4(1+q)}\right),\quad j=0,1,2,
\\  
R^-_{\lambda,3}(q) &= (-1)^{\lambda}\frac{q^{-1/8}}{(1-q^{-1/2})^2}
\frac{q^{\lambda/2}}{1-q},
\\  
R^-_{\lambda,4}(q) &=1 +  \frac{q^{\lambda+2}}{(1+q)^3(1-q)^2},
\\  
R^-_{\lambda,5}(q) &=\frac{q^{-1/8}}{(1+q^{-1/2})^2} \frac{q^{\lambda/2}}{1-q}.
\\  
\end{aligned}
\ee

Thus,
\be
\label{WR}
W_\lambda(q) = R_\lambda(q) + O(q^{3\lambda/2}) \,,
\ee
where $R_\lambda(q) = (R_{\lambda+i,j}(q))_{0 \leq i,j \leq 5}$. 
Since 
\be
\det(W_\lambda(q)) +  O(q^{3\lambda/2})
= \det(R_\lambda(q)) + O(q^{3\lambda/2})= 32q^{\lambda+11/4} + O(q^{3\lambda/2})
\ee
Equation~\eqref{detW-Oq} follows. It is noteworthy that the Eisenstein series
$\calE_2(q)$ which appear in the entries of $R_\lambda(q)$ cancel upon taking
the determinant. The same happens in the entries of the matrix~\eqref{WDR2} below.

This concludes the proof of~\eqref{detW}. We next prove the orthogonality
property \eqref{eq:orthogonality} following the method of \cite[Sec 2.5]{GW:qmod}.
By the q-difference equation \eqref{quadrel}, we have 
\begin{equation}
\label{eq:mat-q-dif}
W_{\lambda+1}(q)=A(\lambda,q)W_\lambda(q),\quad
W_{-\lambda-1}(q^{-1}) =\tilde{A}(\lambda,q)W_{-\lambda}(q^{-1}),
\end{equation}
where
\be
\label{Amat}
A(\lambda,q)=\left(\begin{matrix}
 0 & 1 & 0 & 0 & 0 & 0\\
 0 & 0 & 1 & 0 & 0 & 0\\
 0 & 0 & 0 & 1 & 0 & 0\\
 0 & 0 & 0 & 0 & 1 & 0\\
  0 & 0 & 0 & 0 & 0 & 1\\
 -q & -2q & 1 & 2(1+q) & q+q^{\lambda+4} & -2
\end{matrix}\right)
\ee
and 
\begin{equation*}
\tilde{A}(\lambda,q)= A(-\lambda-1,q^{-1})^{-1}= \left(\begin{matrix}
 -2 & q & 2(1+q) & 1+q^{\lambda-2} & -2q & -q\\
 1 & 0 & 0 & 0 & 0 & 0\\
 0 & 1 & 0 & 0 & 0 & 0\\
 0 & 0 & 1 & 0 & 0 & 0\\
  0 & 0 & 0 & 1 & 0 & 0\\
 0 & 0 & 0 & 0& 1 & 0
\end{matrix}\right).
\end{equation*}
Consider
\begin{equation*}
Q(\lambda,q)\coloneqq\left(\begin{matrix}
 -12 & 8 & -4 & 2 & 0 & 0\\
 8 & -4 & 2 & 0 & 0 & 0\\
 -4 & 2 & 0 & 0 & 0 & 2\\
 2 & 0 & 0 & 0 & 2 & -4\\
  0 & 0 & 0 & 2 & -4 & 8+2q^{\lambda+2}\\
 0 & 0 & 2 & -4& 8+2q^{\lambda+3} & -12-4q^{\lambda+2}-4q^{\lambda+3}
\end{matrix}\right) \,.
\end{equation*}
It is easy to see that the matrices $A$, $Q$ and $\tilde{A}$ (all with entries in the
polynomial ring $\BQ[q^{\pm 1}, q^{\pm \lambda}]$) satisfy
\begin{equation}
\label{eq:quad-q-dif}
A(\lambda,q)Q(\lambda,q)\tilde{A}(\lambda+5,q) = Q(\lambda+1,q) \,.
\end{equation}
Note that all matrices above are invertible, with determinants
\be
\label{detsA}
\det(A(\lambda,q) = q, \qquad
\det(\tilde{A}(\lambda,q) = q, \qquad
\det(Q(\lambda,q) =  -64q^{5+2\lambda} \,.
\ee
Using ~\eqref{eq:mat-q-dif} and ~\eqref{eq:quad-q-dif}, we see that
\begin{equation*}
  W_{\lambda+1}(q)^{-1}Q(\lambda+1,q)\left(W_{-\lambda-6}(q^{-1})^{-1}\right)^T
  = W_{\lambda}(q)^{-1}Q(\lambda,q)\left(W_{-\lambda-5}(q^{-1})^{-1}\right)^T,
\end{equation*}
hence $W_{\lambda}(q)^{-1}Q(\lambda,q)\left(W_{-\lambda-5}(q^{-1})^{-1}\right)^T$
is independent of $\lambda$. The claim is that we have
\begin{equation}
  \label{eq:quad-claim}
  W_{\lambda}(q)^{-1}Q(\lambda,q)\left(W_{-\lambda-5}(q^{-1})^{-1}\right)^T
  = D\coloneqq\left(\begin{matrix}
  0& 0 & \frac{1}{2} & 0 & 0 & 0\\
 0 & 1 &  0& 0 & 0 & 0\\
 \frac{1}{2} & 0 & 0 & 0 & 0 & 0\\
 0 & 0 & 0 & -1 & 0 & 0\\
  0 & 0 & 0 & 0 & \frac{1}{4} & 0\\
 0 & 0 & 0 & 0& 0 & -1
\end{matrix}\right).
\end{equation}

Since we have seen that the left-hand side of ~\eqref{eq:quad-claim} is
independent of $\lambda$, it suffices to show that 
\begin{equation}
  \label{eq:quad-d}
  W_{\lambda}(q)^{-1}Q(\lambda,q)\left(W_{-\lambda-5}(q^{-1})^{-1}\right)^T
  = D + O(q^{\lambda/2}),
\end{equation}
for any sufficiently large $\lambda\in \mathbb{N}$. Equation~\eqref{WR}, together
with~\eqref{detW} gives that
\begin{small}
\be
\label{WDR}
W_{\lambda}(q)^{-1}Q(\lambda,q)\left(W_{-\lambda-5}(q^{-1})^{-1}\right)^T
+ O(q^{\lambda/2})
= R_{\lambda}(q)^{-1}Q(\lambda,q)\left(R_{-\lambda-5}(q^{-1})^{-1}\right)^T
+ O(q^{\lambda/2})
\ee
\end{small}
and an explicit calculation shows that
\be
\label{WDR2}
R_{\lambda}(q)^{-1}Q(\lambda,q)\left(R_{-\lambda-5}(q^{-1})^{-1}\right)^T
+ O(q^{\lambda/2}) = D + O(q^{\lambda/2}) 
\ee
where $R_{-\lambda-5}(q^{-1})^{-1} + O(q^{\lambda/2})$ can be computed by multiplying
the adjugate of $R_{-\lambda-5}(q^{-1}) + O(q^{\lambda/2})$ with the inverse of its
determinant ~\eqref{WR}. Equation~\eqref{eq:quad-d} follows.

%% see text6.tex for the detailed proof omitted.

This concludes the proof of Theorem~\ref{thm.3}.
\qed

%%%%%%%%%%%%%%%%%%%%%%%%%%%%%%%%%%%%%%%%%%%%%%%%%%%%%%%%%%%%%%%%%%%%%%%%%%%%
%%%%%%%%%%%%%%%%%%%%%%%%%%%%%%%%%%%%%%%%%%%%%%%%%%%%%%%%%%%%%%%%%%%%%%%%%%%%

\section{Stationary phase of the descendant state integral}
\label{sec.asy}

\subsection{Stationary phase}
\label{sub.statphase}

In this section we compute the stationary phase of the
state integral around its critical points. This is a well-known method
of asymptotic analysis that can be found in many classic books (eg.,~\cite{Olver}).
For convenience, we define a renormalized version of the descendant
state integral~\eqref{Zdesc} given by

\be
\label{Zhatdesc}
\hat{Z}_{(-2,3,7)}^{(\lambda,\lambda' )}(\tau) \=
(\tq/q)^{\frac{1}{24}} Z_{(-2,3,7)}^{(\lambda,\lambda' )}(\tau) \,.
\ee
Throughout this section, we will use the notation
\be
\label{htaub}
\hbar\coloneqq 2\pi \ramuno \tau, \qquad \tau=b^2 \,,
\ee
and determine the asymptotic expansion of
$\hat{Z}_{(-2,3,7)}^{(\lambda,\lambda' )}(\tau) $ as $\hbar\to 0$. 

It turns out that there are $6$ critical points $\alpha$ 
\begin{equation}
\label{6critpoints}  
(\alpha^3-\alpha-1)(\alpha^3+2\alpha^2-\alpha-1) =0 
\end{equation}
in two Galois orbits of the cubic number fields with discriminants $-23$ and $49$,
respectively. 
% The roots of $x^3+x^2-2x-1$ are $2\cos \left(\frac{2k\pi}{7}\right)$, $k=1,2,3$. 
After a change of parametrization of these number fields (to match with
the conventions of~\cite{GZ:kashaev}, these critical points are given by

\begin{subequations}
\begin{align}
\label{alpha23}
\alpha &= -\xi + \xi^2, \qquad \xi^3-\xi^2+1 = 0 \\
\label{alpha49}
\alpha &= -1 -\eta,     \qquad \eta^3+\eta^2-2\eta-1 =0 \,.
\end{align}
\end{subequations}

The next theorem computes the stationary phase expansion of
$\hat{Z}_{(-2,3,7)}^{(\lambda,\lambda' )}(\tau)$ at each critical point.

\begin{theorem}
\label{thm.5}
The stationary phase of $\hat{Z}_{(-2,3,7)}^{(\lambda,\lambda' )}(\tau)$ is
given by
$\eu^{\frac{2\pi \ramuno\lambda'\log\alpha}{\hbar}} \hat\Phi^{(\alpha)}(\lambda,\hbar)$,
where
\be
\label{eq:phi-gene}
\hat\Phi^{(\alpha)}(\lambda,\hbar) \= \eu^{\frac{V_{0,0}(\a)}{\hbar}}
\Phi^{(\alpha)}(\lambda,\hbar), \qquad \Phi^{(\alpha)}(\lambda,\hbar) \= 
\frac{\alpha^\lambda}{\sqrt{\ramuno \Delta(\a)}} 
  \sum_{k=0}^\infty c_k(\a,\lambda) \hbar^k 
\ee
and 
\begin{equation}
\label{eq:v00-delta}
\begin{aligned}
V_{0,0}(\alpha) & \= 2\Li_2(-\alpha)- \Li_2(\alpha^{-2}),\\
\Delta(\alpha) & \=
%\frac{-2 \alpha^2(\alpha^2-\alpha+2)}{(\alpha-1)^2(\alpha+1)^2}
-2\alpha^5+12\alpha^3 -2\alpha^2-16\alpha-10 \,,
\end{aligned}
\end{equation}
%\be
%\label{eq:v00-delta}
%V_{0,0}(\alpha) \= 2\Li_2(-\alpha)- \Li_2(\alpha^{-2}), \qquad
%\Delta(\alpha)  \=
%%\frac{-2 \alpha^2(\alpha^2-\alpha+2)}{(\alpha-1)^2(\alpha+1)^2}
%-2\alpha^5+12\alpha^3 -2\alpha^2-16\alpha-10 \,.
%\ee
and $c_k(\a,\lambda) \in \BQ(\a)[\lambda]$ are polynomials in $\lambda$ of degree
$2k$ with coefficients in $\BQ(\a)$ with $c_0(\a,\lambda)=1$ given explicitly by a
formal Gaussian integration.
\end{theorem}
  
We have computed 400 coefficients of the above series for $\lambda=0$ and
40 coefficients for general $\lambda$. Since there are two number fields
involved, we present the asymptotic series $\hat{\Phi}^{(\sigma)}(\lambda,\hbar)$
separately for each field. For $\alpha$ as in~\eqref{alpha23}, we have
\begin{small}
\begin{equation}
\label{phi23la}    
\begin{aligned}
  \hat{\Phi}^{(\sigma)}(\lambda,\hbar)
  =\frac{
\alpha^\lambda\eu ^{\frac{V_{0,0}}{\hbar}}}{\sqrt{\ramuno ( -6\xi^2+10\xi-4) }}
  \Bigg( 1 +& \Bigg(\left(-\frac{1}{46}\xi^2 -\frac{7}{92}\xi
+\frac{3}{92}\right)\lambda^2 + \left(\frac{3}{46}\xi^2 -\frac{11}{92}\xi
+\frac{17}{46}\right)\lambda\\
  &+\frac{293}{8464}\xi^2 +\frac{127}{2116}\xi -\frac{681}{8464}\Bigg) \hbar
  +O(\hbar^2)\Bigg),
\end{aligned}
\end{equation}
\end{small}
and for $\alpha$ as in~\eqref{alpha49}, we have
\begin{small}
\begin{equation}
\label{phi49la}    
\begin{aligned}
  \hat\Phi^{(\sigma)}(\lambda,\hbar) = \frac{\alpha^\lambda
 \eu ^{\frac{V_{0,0}}{\hbar}}}{\sqrt{\ramuno(-4\eta^2 +2\eta -2)}}\Bigg(
   1 +&\Bigg( \left(\frac{1}{28}\eta^2 + \frac{1}{14}\eta -\frac{1}{28}\right)
   \lambda^2 + \left(\frac{1}{28}\eta^2 -\frac{1}{14}\eta +\frac{3}{14}\right)\lambda \\
+& \frac{1}{16}\eta^2 + \frac{1}{16}\eta -\frac{17}{168}\Bigg)\hbar +O(\hbar^2)\Bigg) \,.
\end{aligned}
\end{equation}
\end{small}

We can give more terms when $\lambda=0$. For $\alpha$ as in~\eqref{alpha23}, we have
\begin{equation}
\label{phi23}    
\begin{aligned}
\hat{\Phi}^{(\sigma)}(0,\hbar) \= &
\frac{\eu ^{\frac{V_{0,0}}{\hbar}}}{\sqrt{\ramuno (-6\xi^2+10\xi-4)} }
\Bigg(1 + \left(\frac{293}{8464}\xi^2 + \frac{127}{2116}\xi - \frac{681}{8464}\right)
\hbar \\
&+ \left(\frac{65537}{6229504}\xi^2 - \frac{50607}{6229504}\xi
  + \frac{2535}{778688}\right)\hbar^2 +O(\hbar^3)\Bigg),
\end{aligned}
\end{equation}
and for $\alpha$ as in~\eqref{alpha49}, we have
\begin{equation}
\label{phi49}    
\begin{aligned}
\hat{\Phi}^{(\sigma)}(0,\hbar) \= &
\frac{\eu ^{\frac{V_{0,0}}{\hbar}}}{\sqrt{\ramuno(-4\eta^2 +2\eta -2)} }
\Bigg(1 + \left(\frac{293}{8464}\xi^2 + \frac{127}{2116}\xi - \frac{681}{8464}\right)
\hbar \\
&+ \left(\frac{65537}{6229504}\xi^2 - \frac{50607}{6229504}\xi
  + \frac{2535}{778688}\right)\hbar^2 +O(\hbar^3)\Bigg),
\end{aligned}
\end{equation}

\subsection{Formal Gaussian integration}

Using the identity~\eqref{Phi2term}
we convert the descendant state integral into the following form,
\begin{align}\label{cancel}
Z^{(\lambda,\lambda')}_{(-2,3,7)}(\hbar) &=\left(\frac{q}{\tilde q}\right)^{-\frac{1}{24}}
\int_{\mathbb{R}+\ramuno \frac{c_b}{2}+\ramuno\varepsilon}\Phi_{b}(x)^2\Phi_{b}(2x-c_b)
\eu^{-\pi \ramuno(2x-c_b)^2 +2\pi (\lambda b - \lambda ' b^{-1})x}\dif x
\\ & =\int_{\mathbb{R}+\ramuno \frac{c_b}{2}+\ramuno\varepsilon}
\frac{\Phi_{b}(x)^2}{\Phi_{b}(-2x+c_b)} \eu^{2\pi (\lambda b - \lambda ' b^{-1})x}\dif x,
\end{align}
and then apply the approximation~\cite[eq.~(65)]{AK}
\begin{equation*}
\Phi_{b}\left(\frac{z}{2\pi b}\right)=\exp\left(\sum_{n=0}^\infty\hbar^{2n-1}
\frac{B_{2n}(1/2)}{(2n)!}\Li_{2-2n}(-\eu^{z})\right).
\end{equation*}

We begin with a change of variables $x\mapsto\frac{z}{2 \pi b}$, so that
\begin{equation*}
\Phi_{b}(x)^2 = \Phi_b\left(\frac{z}{2\pi b}\right)^2 \sim
\exp\left(\sum_{n=0}^\infty\hbar^{2n-1} \frac{B_{2n}(1/2)}{(2n)!}2\Li_{2-2n}(-\eu^{z})
\right) %,\quad(\hbar =2\pi \ramuno b^2)
\end{equation*}
and
\begin{equation*}
\begin{aligned}
\Phi_{b}(-2x+c_b) = \Phi_{b}\left(\frac{-2z+2\pi b c_b}{2\pi b}\right) &
\sim\exp\left(\sum_{n=0}^\infty\hbar^{2n-1} \frac{B_{2n}(1/2)}{(2n)!}
\Li_{2-2n}(-\eu^{-2z+2\pi b c_b})\right) \\
&=\exp\left(\sum_{n=0}^\infty\hbar^{2n-1} \frac{B_{2n}(1/2)}{(2n)!}
\Li_{2-2n}(\eu^{-2z+\frac{\hbar}{2}})\right). 
\end{aligned}
\end{equation*}
Using the identity
\begin{equation*}
\Li_{2-2n}(\eu^{-2z+s})=\sum_{k=0}^{\infty}\frac{\Li_{2-2n-k}(\eu^{-2z})}{k!}s^k,
\end{equation*}
we have
\begin{equation*}
\Li_{2-2n}(-\eu^{-2z+\frac{\hbar}{2}}) = \sum_{k=0}^{\infty}
\frac{\Li_{2-2n-k}(\eu^{-2z})}{k!}\left(\frac{\hbar}{2}\right)^k.
\end{equation*}
Collecting the above equalities up, we obtain
\begin{equation*}
Z^{(\lambda,\lambda')}_{(-2,3,7)}(\hbar) \sim \frac{\ramuno }{
  \sqrt{2\pi \ramuno \hbar}} \int \exp\left(\lambda z
  + \frac{2\pi \ramuno\lambda'}{\hbar}z+ V(z,\hbar)\right) \dif z,
\end{equation*}
where
\begin{equation*}
\begin{aligned}
V(z,\hbar) =  &\sum_{n=0}^\infty\hbar^{2n-1} \frac{B_{2n}(1/2)}{(2n)!}2
\Li_{2-2n}(-\eu^{z}) - \sum_{n,k\geq 0}\hbar^{2n+k-1} \frac{B_{2n}(1/2)}{(2n)!}
  \frac{\Li_{2-2n-k}(\eu^{-2z})}{k!}\frac{1}{2^k}\\
  = &\sum_{n=0}^\infty\hbar^{2n-1} \frac{B_{2n}(1/2)}{(2n)!}2\Li_{2-2n}(-\eu^{z})
  - \Bigg(\sum_{n,k\geq 0} \hbar ^{2n+2k-1}\frac{B_{2n}(1/2)}{(2n)!}
   \frac{\Li_{2-2n-2k}(\eu^{-2z})}{(2k)!}\frac{1}{2^{2k}}
   \\
   &+ \sum_{n,k\geq 0}\hbar ^{2n+2k}\frac{B_{2n}(1/2)}{(2n)!}
   \frac{\Li_{1-2n-2k}(\eu^{-2z})}{(2k+1)!}\frac{1}{2^{2k+1}} \Bigg) \\
 = & \sum_{n=0}^\infty \hbar ^{2n-1}\left(\frac{B_{2n}(1/2)}{(2n)!}
   2\Li_{2-2n}(-\eu^{z})   -  \sum_{k=0}^n
   \frac{B_{2n-2k}(1/2)}{(2n-2k)!(2k)!}\frac{\Li _{2-2n}(\eu^{-2z})}{2^{2k}}\right)\\
 &+\sum_{n=0}^\infty \hbar ^{2n}\left( - \sum_{k=0}^n \frac{B_{2n-2k}(1/2)}{
 (2n-2k)!(2k+1)!}\frac{\Li_{1-2n}(\eu^{-2z}) }{2^{2k+1}}\right).  
 \end{aligned}
\end{equation*}
Therefore, if we define
\begin{equation}\label{eq:V-n}
\begin{aligned}
  V_{2n+1}(z)& = - \sum_{k=0}^\infty \frac{B_{2n-2k}(1/2)}{(2n-2k)!(2k+1)!}
  \frac{\Li_{1-2n}(\eu^{-2z}) }{2^{2k+1}}, \\
  V_{2n}(z) & = \frac{B_{2n}(1/2)}{(2n)!}2\Li_{2-2n}(-\eu^{z})
  -  \sum_{k=0}^n \frac{B_{2n-2k}(1/2)}{(2n-2k)!(2k)!}
  \frac{\Li _{2-2n}(\eu^{-2z})}{2^{2k}},
\end{aligned}
\end{equation}
then $V(z,\hbar)  = \sum_{n=0}^\infty \hbar ^{n-1}V_n(z)$, hence
\begin{equation*}
  \hat{Z}^{(\lambda,\lambda')}_{(-2,3,7)}(\hbar)\sim
  \frac{\ramuno}{ \sqrt{2\pi \ramuno \hbar}}
  \int\exp\left(\lambda z + \frac{2\pi \ramuno\lambda'}{\hbar}z
    +\sum_{n=0}^\infty \hbar ^{n-1} V_{n}(z)\right)\dif z.
\end{equation*}

Solving $\od{}{z}\left(2\pi \ramuno \lambda'z+V_0(z)\right)=0$, we find that 
the critical point equation is
\begin{equation}
\label{critpoints}  
(\alpha^3-\alpha-1)(\alpha^3+2\alpha^2-\alpha-1) =0, \quad (\alpha = \eu ^z).
\end{equation}
The expansion $V_n(z) = \sum_{m=0}^\infty (z-\log\alpha)^m V_{n,m}(\log\alpha)$ at a
critical point $z= \log\alpha 
%\in \{x\mid (x^3-x-1)(x^3+2x^2-x-1)=0 \}
$ thus gives
\small
\begin{equation}
  \label{eq:asyp-til-gauss}
\begin{aligned}
  \hat{Z}^{(\lambda,\lambda')}_{(-2,3,7)}(\hbar)&\sim
  \frac{\ramuno\alpha^{\lambda}
    \eu ^{\frac{V_{0,0}+2\pi \ramuno\lambda'\log\alpha}{\hbar}}}{\sqrt{2\pi\ramuno } }
  \int\dif y \eu ^{V_{0,2}y^2} \exp\left(\lambda\hbar^{\frac{1}{2}}y+\sum_{m\geq 3}
    \hbar^{\frac{m}{2}-1}
y^m V_{0,m} +\sum_{n\geq 1,m\geq 0}\hbar^{n-1+\frac{m}{2}} y^m V_{n,m}\right)\\
&\eqqcolon \eu^{\frac{2\pi \ramuno\lambda'\log\alpha}{\hbar}} \hat\Phi(\lambda,\hbar).
\end{aligned}
\end{equation}
\normalsize
where the change of variables $z \mapsto \log\alpha +\hbar^{\frac{1}{2}}y$ is applied,
and 
\small
\begin{equation}\label{eq:V-n-m}
\begin{aligned}
 V_{0,0} &= 2\Li_2(-\alpha)-\Li_2(\alpha^{-2}), \\
 V_{0,1} &= - 2\pi \ramuno \lambda',\\
 V_{1,0} &= -\frac{1}{2}\Li_1 (\alpha^{-2}) = \frac{1}{2}\log( 1-\alpha^{-2}),\\
 V_{0,2} &= \Li_0(-\alpha) -  2\Li_ 0 (\alpha^{-2})
 %= \frac{-\alpha}{1+\alpha}-\frac{2}{\alpha^2 - 1}
 = - \frac{\alpha^2 - \alpha +2}{(\alpha-1)(\alpha+1)} 
 =\alpha^5-\alpha^4-7\alpha^3 +\alpha^2+4\alpha+5, \\
 V_{2n,m} &= \frac{1}{m!}\left( \frac{B_{2n}(1/2)}{(2n)!}
   2\Li_{2-2n-m}(-\alpha) -(-2)^m\Li_{2-2n-m}(\alpha^{-2})
   \sum_{k=0}^n \frac{B_{2n-2k}(1/2)}{(2n-2k)!(2k)!2^{2k}}\right),\\
 V_{2n+1,m} &= -\frac{(-2)^m}{m!}\Li_{1-2n-m}(\alpha^{-2})
 \sum_{k=0}^n \frac{B_{2n-2k}(1/2)}{(2n-2k)!(2k+1)!2^{2k+1}}.
\end{aligned}
\end{equation}
\normalsize

Note that for $n=1$ and $m=0$, we have $\hbar^{n-1+\frac{m}{2}} y^m V_{n,m}= V_{1,0}$.
Expand the exponential in the integrand, collect $\hbar$'s and use the formal
Gaussian integrals, we obtain
\begin{equation*}
\hat\Phi(\lambda,\hbar) = \frac{\alpha^{\lambda}
\eu ^{\frac{V_{0,0}}{\hbar}}\eu ^{V_{1,0}}}{
\sqrt{2\ramuno  V_{0,2}} }  \left( 1 +O(\hbar) \right) =
\frac{\alpha^{\lambda}
\eu ^{\frac{V_{0,0}}{\hbar}}}{\sqrt{\ramuno \Delta }}\left( 1 +O(\hbar) \right)  
\end{equation*}
where
\begin{equation*}
  \Delta \coloneqq \frac{2 V_{0,2}}{ \eu ^{2V_{1,0}}} =
  \frac{-2 \alpha^2(\alpha^2-\alpha+2)}{(\alpha-1)^2(\alpha+1)^2}
  \= -2\alpha^5+12\alpha^3 -2\alpha^2-16\alpha-10 \,.   
\end{equation*}
This concludes the proof of Theorem~\ref{thm.5}.
\qed

When $\alpha$ satisfies~\eqref{alpha23}, we have
\be
\label{someV23}
V_{1,0} = % \frac{1}{2}\log( 1-\alpha^{-2})=
\frac{1}{2}\log (1-\xi^2), \qquad
V_{0,2} = % -\frac{\alpha^2 - \alpha +2}{(\alpha-1)(\alpha+1)}=
-3\xi^2 + 2\xi, \qquad
\Delta = -6\xi^2+10\xi-4 
\ee
whereas when $\alpha$ satisfies~\eqref{alpha49}, we have
\be
\label{someV49}
V_{1,0} = %\frac{1}{2} \log( 1-\alpha^{-2}) =
\frac{1}{2}\log(\eta^2+\eta-2), \qquad
V_{0,2} = %-\frac{\alpha^2 - \alpha +2}{(\alpha-1)(\alpha+1)} =
-\eta^2-3\eta+3, \qquad
\Delta = -4\eta^2 +2\eta -2 \,.
\ee
Computing out the formal Gaussian integrals in ~\eqref{eq:asyp-til-gauss}, we
obtain~\eqref{phi49la} and~\eqref{phi49}. 

%%%%%%%%%%%%%%%%%%%%%%%%%%%%%%%%%%%%%%%%%%%%%%%%%%%%%%%%%%%%%%%%%%%%%%%%%%%%
%%%%%%%%%%%%%%%%%%%%%%%%%%%%%%%%%%%%%%%%%%%%%%%%%%%%%%%%%%%%%%%%%%%%%%%%%%%%

\section{Analytic aspects}
\label{sec:asy}

In this last section we discuss analytic aspects of matrix-valued holomorphic
quantum modular forms. We have already introduced the descendant state integrals
which are analytic functions of $\tau \in \BC'$, as well as the collection of
$q$-series $H_j^\pm(q)$ which are holomorphic functions of $\tau$ (where
$q=e^{2\pi i \tau}$) in the upper half plane $\Im(\tau)>0$. In this section we discuss
the radial asymptotics of these holomorphic functions as $\tau$ tends to zero in
a fixed ray  (i.e., $\arg(\tau) =\theta \in (-\pi,\pi)$ is fixed). Naturally, one
expects these asymptotics to be given in terms of the formal power series
$\hat\Phi^{(\a)}(\lambda,\hbar)$ of Section~\ref{sec.asy}, up to some elementary
constants.

All results that we report in this section are numerical, and void of proofs. 

\subsection{Asymptotic expansion of the $q$-series}
\label{sec:asyq}

Fixing a ray $\arg(\tau)=\theta$, we first computed numerically the values of
the series~\eqref{hH} when $\tau=e^{i \theta}/N$ for $N=800,\dots,1000$ to high
precision in \texttt{pari} using the inductive definition of $p_{0,0,m}(q)$ and
$P_{0,0,n}(q)$ obtained easily from their definition~\eqref{p012def}. 

We then used Richardson and Zagier's extrapolation methods which
are explained in \cite{GZ:kashaev} and in great detail in \cite{Wheeler:thesis},
to extrapolate numerically from this data the coefficients of their asymptotic
expansion. Properly normalized, these are algebraic numbers in a known number field
(one of the cubic fields of~\eqref{alpha23}--~\eqref{alpha49}) that are known to
high accuracy, which can then be recognized exactly. Having done so, the coefficients
that we found ought to match one of the $\hat\Phi^{(\s)}(\hbar)$ series, up to
some elementary factors, for some value of $\s$, which of course depends on the ray.

For example, when $\arg(\tau) = \pi/5$, we found numerically the following relation
between the radial asymptotics of the $q$-series~\eqref{hH} and
the asymptotic series~$\hat\Phi$, where $\hbar=2 \pi i \tau$ and $q=e^{2 \pi i \tau}$. 

\be
\label{qasy}
\begin{aligned}
H_{0,0}^+(q) &\= \left(\frac{q}{\tilde{q}}\right)^{1/24}\tau\eu^{\frac{\pi \ramuno}{4}}
\hat\Phi^{(\sigma_1)}(\hbar), &
H_{0,0}^-(q) &\= \left(\frac{q}{\tilde{q}}\right)^{-1/24}\tau\eu^{\frac{\pi \ramuno}{4}}
\hat\Phi^{(\sigma_2)}(-\hbar) \\
H_{0,1}^+(q) &\= \left(\frac{q}{\tilde{q}}\right)^{1/24}\eu^{\frac{\pi \ramuno}{4}}
\hat\Phi^{(\sigma_1)}(\hbar) &
H_{0,1}^-(q) &\= \left(\frac{q}{\tilde{q}}\right)^{-1/24}\eu^{\frac{\pi \ramuno}{4}}
\hat\Phi^{(\sigma_2)}(-\hbar) \\
H_{0,2}^+(q) &\= \left(\frac{q}{\tilde{q}}\right)^{1/24}\frac{2}{3\tau}
\eu^{\frac{\pi \ramuno}{4}}\hat\Phi^{(\sigma_1)}(\hbar) &
H_{0,2}^-(q) &\= \left(\frac{q}{\tilde{q}}\right)^{-1/24}\frac{5}{6\tau}
\eu^{\frac{\pi \ramuno}{4}}\hat\Phi^{(\sigma_2)}(-\hbar) \\
H_{0,3}^+(q) &\= \left(\frac{q}{\tilde{q}}\right)^{1/24}\frac{1}{2}\eu^{-\frac{\pi \ramuno}{4}}\hat\Phi^{(\sigma_1)}(\hbar) &
H_{0,3}^-(q) &\= \left(\frac{q}{\tilde{q}}\right)^{-1/24} \frac{1}{2}\eu^{-\frac{\pi \ramuno}{4}}
\hat\Phi^{(\sigma_2)}(-\hbar) \\
H_{0,4}^+(q) &\= \tilde q ^{-\frac{7}{8}} \left(\frac{q}{\tilde{q}}\right)^{1/24}
2\eu^{-\frac{\pi \ramuno}{4}}\hat\Phi^{(\sigma_6)}(\hbar) &
H_{0,4}^-(q) &\= \tilde q ^{\frac{7}{8}} \left(\frac{q}{\tilde{q}}\right)^{-1/24}
2\eu^{-\frac{\pi \ramuno}{4}}\hat\Phi^{(\sigma_3)}(\hbar) \\
H_{0,5}^+(q) &\= \tilde q ^{-\frac{7}{8}} \left(\frac{q}{\tilde{q}}\right)^{1/24}
\eu^{-\frac{\pi \ramuno}{4}}\hat\Phi^{(\sigma_6)}(\hbar) &
H_{0,5}^-(q) &\= \tilde q ^{\frac{7}{8}} \left(\frac{q}{\tilde{q}}\right)^{-1/24}
\eu^{-\frac{\pi \ramuno}{4}}\hat\Phi^{(\sigma_3)}(\hbar) \,.
\end{aligned}
\ee

Here $\s_j$ for $j=1,\dots,6$ are the six roots of the polynomial~\eqref{6critpoints}
with the numerical values
\be
\label{s123}
\sigma_1 =-0.662-0.562 i, \qquad
\sigma_2=-0.662+0.562 i, \qquad
\sigma_3=1.325
\ee
corresponding to the field~\eqref{alpha23} and
\be
\label{s456}
\sigma_4 =-2.247, \qquad
\sigma_5=-0.555, \qquad
\sigma_6=0.802 \,,
\ee
corresponding to the field~\eqref{alpha49}, respectively.

Note that inserting the asymptotics~\eqref{qasy} to the quadratic
relation~\eqref{quadrel}, one simply obtains that $0=0$. 

%% pari files: programs/pari/237-Pretzel: 237.q-series  AllThreeKnots.stavros.q-series

\subsection{Further aspects}
\label{sub.further}

As we explained briefly in the introduction, matrix-valued holomorphic quantum
modular forms are complicated objects with conjectural analytic and arithmetic
properties. In the present paper, we focused on the algebraic aspects of these
objects. Our paper does not include the following analytic aspects of
the matrix-valued holomorphic quantum modular form of the $(-2,3,7)$-pretzel knot:
\begin{itemize}
\item
  Asymptotics of the state integral as $\tau$ tends to zero in a fixed ray.
  A detailed analysis of the corresponding thimbles will surely identify those
  asymptotics with $\BZ$-linear combinations of the series
  $\tq^{\lambda'}\hat\Phi(\lambda,\hbar)$.
\item
  Borel resummation of the factorially divergent series $\hat\Phi(\lambda,\hbar)$,
  and identification of their Stokes phenomenon in terms of the series
  $H_{\lambda,j}(q)$. Without doubt, the Borel resummation coincides, up to elementary
  factors, with the descendant state integral itself. 
\item
  Asymptotics of the $q$-series $H_{\lambda,j}(q)$ when $\tau$ tends to zero
  in a fixed ray. This can be deduced combining Theorem~\ref{thm.4} with the
  asymptotics of the state integrals themselves.
\end{itemize}

The paper also not include the arithmetic aspects related to the matrix of Habiro-like
elements. Those can be obtained by the factorization of the descendant state integral
~\eqref{Zdesc} at rational points, following~\cite{GK:evaluation}. 

\subsection*{Acknowledgements}

The authors would like to thank Zhihao Duan, Jie Gu and especially Campbell Wheeler
for enlightening conversations. 
\vspace{0.5cm}

\noindent
{\bf Conflict of interest.} To the best of all authors’ knowledge, the submitted
article has no conflict of interest.
 
%%%%%%%%%%%%%%%%%%%%%%%%%%%%%%%%%%%%%%%%%%%%%%%%%%%%%%%%%%%%%%%%%%%%%%%%%%%% 
%%%%%%%%%%%%%%%%%%%%%%%%%%%%%%%%%%%%%%%%%%%%%%%%%%%%%%%%%%%%%%%%%%%%%%%%%%%%

\appendix

\section{Comparison of the $q$-series with~\cite{GZ:qseries}}
\label{sub.comparison}

Recall the $q$-series $H^+_k(q)$ and $H^-_k(q)$ for $|q| < 1$
from Equations (142) and (143) of~\cite{GZ:qseries}. When $k=0,1,2$,
the series $H^\pm_k(q)$ coincide with $H^\pm_{0,k}(q)$, whereas when $k=3,4,5$, they
are given by
{\tiny
\be
\label{eq.237Hb} 
\begin{aligned}
H^+_3(q) & \= \frac{(q^{3/2};q)_\infty^2}{(q;q)_\infty^2}
\sum_{m=0}^\infty \frac{q^{(2m+1)(m+1)}}{(q^{3/2};q)_m^2 (q;q)_{2m+1}}
& H^-_4(q) & \= \frac{(q;q)_\infty^2}{(-1;q)_\infty^2} 
\sum_{n=0}^\infty \frac{q^{n(n+1)}}{(-q;q)_n^2 (q;q)_{2n}}
\\
H^+_4(q) & \= \frac{(-q;q)_\infty^2}{(q;q)_\infty^2}
\sum_{m=0}^\infty \frac{q^{(2m+1)m}}{(-q;q)_m^2 (q;q)_{2m}}
& H^-_3(q) & \= \frac{(q;q)_\infty^2}{(q^{-1/2};q)_\infty^2} 
\sum_{n=0}^\infty \frac{q^{n(n+2)}}{(q^{3/2};q)_n^2 (q;q)_{2n+1}}
\\
H^+_5(q) & \= \frac{(-q^{3/2};q)_\infty^2}{(q;q)_\infty^2}
\sum_{m=0}^\infty \frac{q^{(2m+1)(m+1)}}{(-q^{3/2};q)_m^2 (q;q)_{2m+1}}
& H^-_5(q) & \= \frac{(q;q)_\infty^2}{(-q^{-1/2};q)_\infty^2} 
\sum_{n=0}^\infty \frac{q^{n(n+2)}}{(-q^{3/2};q)_n^2 (q;q)_{2n+1}} \,.
\end{aligned}
\ee
}

%% see pari file:
%% programs/pari/zagier/qKnots/qSeries/237pretzel/AllThreeKnots.stavros.q-series
%% and Dropbox/qKnots/qSeries/Index/AllThreeKnotsNew.q.stavros

The comparison between the above series with the ones in our paper are
given as follows.

\begin{lemma}
\label{lem.compare}
We have:
\be
\label{eq.comparison} 
\begin{aligned}
H^+_{0,3}(q) & \= \frac{q^{1/8}}{(1-q^{1/2})^2}
\frac{(q;q)_\infty^2}{(q^{3/2};q)_\infty^2} H^+_3(q) &   
H^-_{0,3}(q) & \= \frac{q^{-1/8}}{(1-q^{-1/2})^2}
\frac{(q^{-1/2};q)_\infty^2}{(q;q)_\infty^2}  H_3^-(q)
\\
H^+_{0,4}(q) & \=  \frac{(q;q)_\infty^2}{(-q;q)_\infty^2} H^+_4(q) &   
H^-_{0,4}(q) & \=  \frac{(-1;q)_\infty^2}{(q;q)_\infty^2}  H_4^-(q)
\\
H^+_{0,5}(q) & \= \frac{q^{1/8}}{(1+q^{1/2})^2}
\frac{(q;q)_\infty^2}{(-q^{3/2};q)_\infty^2} H^+_5(q) &   
H^-_{0,5}(q) & \= \frac{q^{-1/8}}{(1+q^{-1/2})^2}
\frac{(-q^{-1/2};q)_\infty^2}{(q;q)_\infty^2}  H_5^-(q)
\end{aligned}  
\ee
\end{lemma}

\begin{proof}
We need to show the following identities:
\begin{equation}
\label{connection}
\begin{aligned}
\frac{(q^{3/2};q)_\infty^2}{(q;q)_\infty^2}\frac{(\tilde{q};\tilde{q})_\infty^2}{
(-1;\tilde{q})_\infty^2}&=
\frac{\eu^{-\frac{\pi \ramuno}{2}}q^{1/8}}{2(1-q^{1/2})^2} \tau,\\
\frac{(-q;q)_\infty^2}{(q;q)_\infty^2}\frac{(\tilde{q};\tilde{q})_\infty^2}{
(-\tilde{q}^{-1/2};\tilde{q})_\infty^2}&=
\frac{\eu^{-\frac{\pi \ramuno}{2}}\tilde q^{-1/8}}{2(1-\tilde{q}^{-1/2})^2} \tau,\\
\frac{(-q^{3/2};q)_\infty^2}{(q;q)_\infty^2}\frac{(\tilde{q};\tilde{q})_\infty^2}{
(-\tilde{q}^{-1/2};\tilde{q})_\infty^2}&=
\frac{\eu^{-\frac{\pi \ramuno}{2}}q^{1/8}\tilde{q}^{-1/8}}{
(1+q^{1/2})^2(1+\tilde{q}^{-1/2})^2}\tau. 
\end{aligned}
\end{equation}
The modularity of the Dedekind $\eta$-function implies that
\be
\frac{(q;q)_\infty}{(\tilde{q};\tq)_\infty}
=e^{\frac{\pi \ramuno}{4}}
\left( \frac{q}{\tilde{q}}\right)^{\frac{1}{24}}\frac{1}{\sqrt{\tau}} \,,
\ee
hence
\begin{equation*}
\begin{aligned}
  \frac{(q^{1/2};q^{1/2})_\infty}{(\tilde{q}^2;{\tilde{q}^2})_\infty}=&e^{\frac{\pi \ramuno}{4}}
  \frac{\tilde{q}^{1/12}}{q^{1/48}}
  \left({\sqrt{\frac{\tau}{2}}}\right)^{-1}.
  %\\\frac{(q^2;q^2)_\infty}{(\tilde{q}^{1/2};{\tilde{q}^{1/2}})_\infty}=&e^{\frac{\pi \ramuno}{4}} \left( \frac{\tilde{q}^{1/2}}{q^2}\right)^{\frac{1}{24}}\left({\sqrt{2}{b}}\right)^{-1}=e^{\frac{\pi \ramuno}{4}} \frac{\tilde{q}^{1/2}}{q^{1/12}}\left(\sqrt{2\tau }\right)^{-1}.
\end{aligned}
\end{equation*}
Since 
\begin{equation*}
    \begin{aligned}
        (q^2;q^2)_\infty=&(q;q)_\infty (-q;q)_\infty\\
        (q^{1/2};q^{1/2})_\infty= &  (q^{1/2};q)_\infty(q;q)_\infty   
    \end{aligned}
\end{equation*} 
it follows that
\begin{equation*}
\begin{aligned}
  \frac{(q^{3/2};q)_{\infty}}{(q;q)_{\infty}}
  \frac{(\tilde{q};\tilde{q})_\infty}{(-1;\tilde{q})_\infty}=&
  \frac{1}{2(1-q^{1/2})}
  \frac{(q^{1/2};q)_{\infty}}{(-\tilde{q};\tilde{q})_{\infty}}
  \frac{(\tilde{q};\tilde{q})_\infty}{(q;q)_\infty}\\
  =& \frac{1}{2(1-q^{1/2})}
  \frac{(q^{1/2};q^{1/2})_{\infty}}
  {(\tilde{q}^2;\tilde{q}^2)_\infty}
  \frac{(\tilde{q};\tilde{q})_\infty^2}{(q;q)_{\infty}^2}
  \\
  =&\frac{1}{2(1-q^{1/2})}
  e^{-\frac{\pi\ramuno}{4}}q^{1/16}\sqrt{2\tau}.
\end{aligned}
\end{equation*}
Therefore 
\begin{equation*}
  \frac{(q^{3/2};q)_\infty^2}{(q;q)_\infty^2}
  \frac{(\tilde{q};\tilde{q})_\infty^2}{(-1;\tilde{q})_\infty^2}=
  \frac{\eu^{-\frac{\pi \ramuno}{2}}q^{1/8}}{2(1-q^{1/2})^2} \tau.
\end{equation*}
The proof for the rest two identities is similar.
\end{proof}

We end this appendix with a remark that the collection of $q$-hypergeometric
series $H^\pm_{\lambda,j}(q)$ defined and convergent for $|q|<1$ extend to $|q|>1$
and satisfy the symmetry 

\begin{equation}
\label{Hsym}
H_{\lambda,j}^+(q^{-1}) = (-1)^{\delta_j}H_{-\lambda,j}^-(q),\quad j=0,1,2 \qquad
(|q| \neq 1)
\end{equation}

This extension is possible since
\begin{itemize}
\item
  the terms of~\eqref{t012def} satisfy $t_{\lambda,m}(q^{-1}) = T_{-\lambda,m}(q)$,
\item
  the Eisenstein series $\mathcal{E}_1(q)$ and $\mathcal{E}_2(q)$ can be extended
  to $|q|>1$ satisfying $\mathcal{E}_j(q)=-\mathcal{E}_j(q^{-1}) $
  for $j=1,2$~\cite[Remark 19]{GK:evaluation},
\item
  consequently, the terms~\eqref{p012def} satisfy
  $p_{k,j,m}(q) = (-1)^j P_{-k,j,m}(q^{-1})$ for $j=0,1,2$. This follows from the
  identities
  \begin{equation}
    E_1^{(0)}(q) = \frac{1-\mathcal{E}_1(q)}{4} \qquad
      E_2^{(0)}(q) = \frac{1-\mathcal{E}_2(q)}{24},
  \end{equation}
  and the recursive relations ~\eqref{recE1m} and ~\eqref{recE2m}.
\end{itemize}
The symmetry for $j=3,4,5$ is obvious.

%% see text6.tex for the detailed proof omitted.

%%%%%%%%%%%%%%%%%%%%%%%%%%%%%%%%%%%%%%%%%%%%%%%%%%%%%%%%%%%%%%%%%%%%%%%%%%%% 
%%%%%%%%%%%%%%%%%%%%%%%%%%%%%%%%%%%%%%%%%%%%%%%%%%%%%%%%%%%%%%%%%%%%%%%%%%%%

\bibliographystyle{hamsalpha}
\bibliography{biblio}
\end{document}